\documentclass[12pt]{amsart}

\usepackage{amsmath, amssymb}
\usepackage{enumitem}
\newtheorem{theorem}{Theorem}[section]
\newtheorem{lemma}[theorem]{Lemma}
\newtheorem{prop}[theorem]{Proposition}
\newtheorem{cor}[theorem]{Corollary}

\theoremstyle{definition}

\theoremstyle{remark}
\newtheorem{remark}[theorem]{Remark}

\numberwithin{equation}{section}

\newcommand{\rr}{{\mathbb R}}

\newcommand{\N}{{\mathbb N}}

\renewcommand{\P}{\mathbb{P}}

\allowdisplaybreaks

\renewcommand{\P}{\mathbb{P}}
\newcommand{\Pcal}{\mathcal{P}}
\newcommand{\E}{\mathbb{E}}

\newcommand{\1}{\mathbf{1}}

\usepackage{xcolor}

\begin{document}
\sloppy
\title[The longest edge of the 1D soft RGG with boundaries]{The longest edge of the one-dimensional soft random geometric graph with boundaries}

\author{Arnaud Rousselle}
\address{Arnaud Rousselle, Institut de Mathématiques de Bourgogne, UMR 5584, CNRS, Université de Bourgogne, F-21000 Dijon, France}
\email{arnaud.rousselle@u-bourgogne.fr}

\author{Ercan S\"onmez}
\address{Ercan S\"onmez, Department of Statistics, University of Klagenfurt, Universit\"atsstraße 65--67, 9020 Klagenfurt, Austria}
\email{ercan.soenmez\@@{}aau.at}

\begin{abstract}
The object of study is a soft random geometric graph with vertices given by a Poisson point process on a line and edges between vertices present with probability that has a polynomial decay in the distance between them. Various aspects of such models related to connectivity structures have been studied extensively. In this paper we study the random graph from the perspective of extreme value theory and focus on the occurrence of single long edges. The model we investigate has non-periodic boundary and is parameterized by a positive constant $\alpha$, which is the power for the polynomial decay of the probabilities determining the presence of an edge. As a main result we provide a precise description of the magnitude of the longest edge in terms of asymptotic behavior in distribution. Thereby we illustrate a crucial dependence on the power $\alpha$ and we recover a phase transition which coincides with exactly the same phases in \cite{BenBer}.
\end{abstract}


\maketitle

\baselineskip=18pt
\sloppy

\section{Introduction}

In this paper we consider the soft random geometric graph \cite{Penrose2}. In this graph the vertex set is given by the points of a homogeneous Poisson point process and edges are added at random between {vertices with probability depending on their distance}. This is a natural extension of the Gilbert graph \cite{gilbert}, also known as the hard random geometric graph, because an edge is added if the distance between its endpoints is below a given threshold and otherwise it is not added. Concrete applications of such random graphs are numerous and range from computer science \cite{gupta}, engineering \cite{amin}, communication theory \cite{tse}, robotic swarming \cite{von}, climate dynamics \cite{donges} and spread of diseases \cite{eubank}. In particular they proved to be relevant because of taking into account realistic features of modern networks of various types. Predominantly the investigations have been concerned with connectivity structures, including component counts, degree counts and graph distances (see \cite{devroye, last2, Penrose3,  s2} for example). Particular attention has been devoted to low-dimensional versions of the soft random geometric graph separately, very recently in dimension 2 \cite{Penrose5} and in dimension 1 \cite{wilsher1,wilsher2}. The model in dimension 1 admits an interesting application in vehicular ad-hoc networks (see \cite{wilsher2} for a detailed discussion) and will be the object of study in this paper.

\subsection*{The model}
Throughout this paper let $n\in\N$, $B_n = [-n,n]$ and let $\Pcal$ be a homogeneous Poisson point process on $\rr$ with unit intensity. We will consider $\Pcal_n := \Pcal \cap B_n$ and view it as a random (countable) subset of $B_n$. The set $\Pcal_n$ will be the vertex set of the soft random geometric graph. Given $\Pcal_n$ the edge-set is constructed as follows. Let $g\colon \rr \to [0,1]$ be a measurable function. We assume that $g$ is symmetric, i.e. $g(x) = g(-x)$ for all $x \in \rr$. Each pair of points $x,y \in \Pcal_n$ will be connected by an edge $\{x,y\}$ with probability $g(x-y)$ and independently of all other pairs of points. All such edges $\{x,y\}$ will be summarized in the edge-set $E_n$. Thus the graph is given by $(\Pcal_n, E_n)$. The function $g$ is called the connection function. Let $\alpha\in (0, \infty)$. {We will mainly consider the case that $g$ satisfies}
\begin{equation}\label{cf}
	\lim_{|x| \to \infty} \frac{g(x)}{|x|^{-\alpha}} =1.
\end{equation}
An explicit example is given by
$$g(x) = 1-\exp (-|x|^{-\alpha}), \quad x \in \rr.$$
In particular $g$ has unbounded support and polynomial decay as $|x| \to \infty$.

{We note that, unlike the model investigated in \cite{devroye} which is defined on a torus, the random graph considered in this paper has boundaries, meaning that the vertex set is defined on $[-n,n]$.}

\subsection*{Problem formulation and main results}

We focus on the length of the longest edge in the soft random geometric graph, formally defined as
\begin{equation}\label{estar}
	e_n^* := \max_{\{x,y\} \in E_n} |x-y|.
\end{equation}
We are interested in the asymptotic behavior of $e_n^*$ as $n \to \infty$. Our findings provide a precise description of the magnitude of the largest edge-length. Here is the statement of the main results, followed immediately after by discussions and interpretations that manifest their significance.


\begin{theorem}\label{main}
	For $n\in \N$ and $\alpha\in (0,\infty)$ let $(\Pcal_n, E_n)$ be the soft random geometric graph with connection function $g$ satisfying \eqref{cf}. Let $e_n^*$ given in \eqref{estar} be the length of the longest edge in this graph {and let $U$ be a $(0,1)$-uniform distributed random variable}. 
	\begin{itemize}
		\item [(I)] If $\alpha>2$ then
		{$$ \exp \left( \frac{2n}{1-\alpha}    (e_n^*)^{1-\alpha} \right) \xrightarrow[n \to \infty]{d}U.$$}
		\item [(II)] If $\alpha=2$ then
		
		{$$ h \left( \frac{e_n^*}{n}\right)^{-1}  \xrightarrow[n \to \infty]{d}U,$$}
		where $h(x)=\frac{x}{2}\exp\left(\frac{2-x}{x}\right)$ for $x\in (0,2]$.
		\item [(III)] If $1<\alpha<2$ then
		{$$\exp \left( -(2n)^{\alpha-2} 2^{-1}(2-\alpha)^{-2}  \left( (e_n^*)^{2-\alpha} - (2n)^{2-\alpha} \right)^2 \right) \xrightarrow[n \to \infty]{d}U.$$}
		\item [(IV)] If $\alpha=1$ then
		{$$\exp\left( -n \left( \ln\frac{e_n^*}{2n} \right) ^{2} \right) \xrightarrow[n \to \infty]{d}U.$$}
		\item [(V)] If $\alpha<1$ then
		{$$\exp \left( -(2n)^\alpha 2^{-1} (1-\alpha)^{-2}  \left( (e_n^*)^{1-\alpha} - (2n)^{1-\alpha} \right)^2 \right) \xrightarrow[n \to \infty]{d}U.$$}
	\end{itemize}
\end{theorem}

\subsection*{Interpretations and discussions}

{{Theorem \ref{main} shows that the largest edge-length after suitable transformation approximates a $(0,1)$-uniform distributed random variable. It thus follows from the well-known inversion method (see \cite[Theorem 2.1]{devb}) that the largest edge-length suitably transformed approximates an arbitrary random variable $Z$.} This enables to spot many particular examples and, additionally, illustrates significant differences stated in this way between each of the phases.

Let $\alpha>2$. The perhaps most natural {consequence of} Theorem \ref{main} is that
$$  \left( \frac{2n}{\alpha-1}  \right)^{\frac{1}{1-\alpha}}  e_n^* \xrightarrow[n \to \infty]{d}Z_{\alpha-1},$$
where $Z_{\alpha-1}$ has a Fr\'echet distribution with parameter $\alpha-1$. That is after scaling the edge-length we obtain a convergence towards a heavy-tailed extreme value distribution, as it is typically the case for maxima of heavy-tailed random variables. It is interesting to compare this statement with the case $\alpha=2$, in which Theorem \ref{main} yields the asymptotics
$$  \frac{1}{2n}e_n^* \xrightarrow[n \to \infty]{d}Z^*,$$
where now $Z^*$ is a random variable with continuous distribution function
$$\mathbb{P} (Z^* \leq r) =  r^{-1} \exp \left( \frac{r-1}{r} \right) \1_{(0,1)} (r) + \1_{[1,\infty)} (r) , \quad r\in\rr,$$
a statement we could not have guessed.

Now let $1<\alpha<2$. The most apparent example we encounter here is a Weibull convergence for a power of the maximum edge-length, namely
$$  2^{-\frac12}(2-\alpha)^{-1}( 2n)^{\frac{\alpha}{2}-1}  \left( (e_n^*)^{2-\alpha} - (2n)^{2-\alpha} \right) \xrightarrow[n \to \infty]{d}Z_{W},$$
where $Z_W$ has a Weibull distribution with parameter 2, i.e.
$$\mathbb{P} (Z_W \leq r) = \exp \left( -(-r)^2 \right) \1_{(-\infty,0]} (r) + \1_{(0,\infty)} (r) , \quad r\in\rr.$$

Another interesting comparison is the one of {\color{black} the latter} statement with an analogous statement in case $\alpha< 1$, which reads
$$  2^{-\frac12}(1-\alpha)^{-1}( 2n)^{\frac{\alpha}{2}}  \left( (e_n^*)^{1-\alpha} - (2n)^{1-\alpha} \right) \xrightarrow[n \to \infty]{d}Z_{W}.$$
{\color{black} We observe that, for $\alpha <1$, $2n-e_n^*$ converges in probability to 0. Since
\begin{align*}
    \frac{1}{1-\alpha}( 2n)^{\frac{\alpha}{2}}  \left( (e_n^*)^{1-\alpha} - (2n)^{1-\alpha} \right)&=\frac{1}{1-\alpha}( 2n)^{1-\frac{\alpha}{2}}  \left( \left(1+\frac{e_n^*-2n}{2n}\right)^{1-\alpha} - 1 \right)\\
    &=( 2n)^{1-\frac{\alpha}{2}}  \left( \left(\frac{e_n^*-2n}{2n}\right)+o\left(\frac{e_n^*-2n}{2n}\right)\right),
\end{align*}
it follows }that \begin{equation}\label{eq:CVW}2^{-\frac12}(2n)^{-\frac{\alpha}{2}}(e^*_n-2n)\xrightarrow[n \to \infty]{d} Z_W.
\end{equation}
{\color{black} A linearization argument as above also applies} in the case $\alpha = 1$ to get \eqref{eq:CVW} but fails for $1<\alpha<2$ since $2n-e_n^*$ does not vanish asymptotically in this case. 

We note that for $\alpha\leq 1$ we have
$$\int_\rr g(x) dx = \infty,$$
whereas for $\alpha>1$ it holds
$$0 < \int_\rr g(x) dx < \infty$$
and the corresponding statements illustrate a crucial change of behavior of the edge-lengths. For example, for $\alpha=\frac{3}{2}$ the assertion is
$$  2^{\frac12} ( 2n)^{-\frac{1}{4}}  \left( \sqrt{e_n^*} - \sqrt{2n} \right) \xrightarrow[n \to \infty]{d}Z_{W},$$
whereas for $\alpha=\frac12$ it is
$$  2^{\frac12}( 2n)^{\frac{1}{4}}  \left( \sqrt{e_n^*} - \sqrt{2n} \right) \xrightarrow[n \to \infty]{d}Z_{W}.$$
Notice here the crucial difference in the order of the sequences scaling the centred quantity $\sqrt{e_n^*} - \sqrt{2n}$. In particular, this shows that for $\alpha=\frac{3}{2}$ the explicit asymptotic behavior of $ \sqrt{e_n^*} - \sqrt{2n}$ can be compared to $2^{-\frac{1}{4}} Z_W n^{\frac{1}{4}}$, whereas for $\alpha=\frac12$ it can be compared to $2^{-\frac{3}{4}} Z_W n^{-\frac{1}{4}}$. We point out that, for $\alpha \leq 1$, almost surely, for each $x\in\mathcal{P}$, the degree of $x$ in $(\mathcal{P}_n, E_n)$ tends to $\infty$ as $n \to \infty$ (which is in contrast to the case $\alpha>1$). Nevertheless, Theorem~\ref{main} allows for an insight and precise description of the growth of the maximum edge-length.  

Finally, as an example of the case $\alpha=1$ one can derive the asymptotics
$$ \left(\frac{e_n^*}{2n}\right)^{\sqrt{n}} \xrightarrow[n \to \infty]{d}Z^{**},$$
where $Z^{**}$ has distribution function
$$\mathbb{P}(Z^{**} \leq r) = r^{-\ln r} \1_{(0,1)} (r) + \1_{[1,\infty)} (r) , \quad r\in\rr.$$

\subsection*{Concluding remarks}
We now give some connections to the literature and earlier works. The choice \eqref{cf} of the connection function naturally admits a comparison with long-range percolation \cite{schulman}. Recall that in classical long-range percolation the vertex set is given by $V=\mathbb{Z}^d$, $d \in \mathbb{N}$, and every pair of points $x,y \in V$ is connected by an edge with probability $p(x-y)$ and independently of all other pairs of points, where $p(x-y) \sim \|x-y\|^{-\alpha}$, as $\|x-y\| \to \infty$. Notably many works on long-range percolation originated from the pioneering paper \cite{BenBer} focusing on graph distances in a finite version of long-range percolation in $d=1$. This has later been continued in arbitrary dimension $d\geq 1$ \cite{CGS} and resulted in a prosperous research branch (see also \cite{Biskup} and references therein). The investigations on graph distances have essentially exposed a phase transition with five distinct regimes of typical behavior, namely $\alpha<d$, $\alpha=d$, $d<\alpha<2d$, $\alpha=2d$, $\alpha>2d$. We would like to point out that this information regarding distinct regimes with different behavior of long-range percolation is reflected in our findings as well, because a phrase transition with precisely the aforementioned parameter regimes appears in our discoveries about the largest edge-length.

Finally we close this section with a few remarks regarding the extension of our results to the soft random geometric graph in arbitrary dimension $d\geq 1$. {One could consider the problem discussed in this paper for the multidimensional generalization of our model. As this model has boundaries, the distribution of the largest edge-length connecting a given vertex depends on the position of this vertex, which we call boundary effects in our model. In dimension $d \geq 2$ this has the consequence that the required computations in order to mimic our results become intractable.} One way to avoid this problem is to define the soft random geometric graph on a $d$-dimensional torus (for example as in \cite{devroye}). Then our methods can be applied in a straightforward way (with essentially less lengthy calculations than the ones in the Appendix), i.e. with the methods provided in this paper one can also prove results for such a multidimensional model. Motivated by the previous remarks in such a model with the same connection function \eqref{cf} we would obtain the phase transition $\alpha<d, \alpha=d, d<\alpha<2d, \alpha=2d, \alpha>2d.$

The rest of the paper is devoted to the proof of Theorem \ref{main}.


\section{Proof of Theorem \ref{main}}

\subsection{Prerequisites}

The first aim of this subsection is to briefly give a definition of the soft random geometric graph in terms of a construction derived from a marked Poisson point process (see {\it e.g.} \cite[Section 4]{Penrose3} for more details). This fact is an essential property of the soft random geometric graph, useful for practical purposes and allows for instance to apply several convergence theorems that enable to investigate and characterize various structures exhibited in the model. This formal construction has been used in various instances, such as in the study of degree counts, component counts and number of edges \cite{Penrose3, last2}. It will permit us to apply particularly a general Poisson approximation theorem via coupling related to Stein's method taken from \cite{Penrose3}, which we previously applied in the same fashion as in the study of the longest edge in the random connection model (see \cite{rs}). We recall a particular case of this result below for the reader's convenience.

\subsubsection{A formal construction of the soft random geometric graph}\label{Sec:ConsSoftRGG} Let us provide a formal construction of the soft random geometric graph from a marked Poisson point process; see \cite{Penrose3} for {an equivalent} construction. We consider as mark space the set $\mathbb{M}=[0,1)^{\mathbb{N}_0}$ and as mark distribution $\mathbf{m}$ the distribution of an infinite sequence of $[0,1)$-uniform independent random variables. Let $n\in\mathbb{N}$ and $\mathcal{P}_n^{\text{m}}$ an independently marked Poisson point process of unit intensity on $B_n$ with marks in $\mathbb{M}=[0,1)^{\mathbb{N}_0}$ distributed according to $\mathbf{m}$. Observe that the event
\[\left\{\begin{array}{c}
     \mbox{there exist distinct marked points } (X,(T_n)_{n\geq 0})\in\mathcal{P}_n^{\text{m}}\mbox{ and }(Y,(U_n)_{n\geq 0})\in\mathcal{P}_n^{\text{m}}\\\mbox{such that }\,T_0=U_0 
\end{array}\right\}\]
has 0 probability.
Thus, almost every realization of this marked Poisson point process can be enumerated as $\{(x_i,(t_{i,0}, t_{i,1}, \dots))\}_{i=1,\dots, |\mathcal{P}_n^{\text{m}}|}$, where $t_{1,0}<t_{2,0}<\dots <t_{|\mathcal{P}_n^{\text{m}}|,0}$. This enumeration presents the advantage to give a canonical way to list the points in order to decide if they are connected or not. More precisely, we define the edge-set $E_n$ as the set of all the edges $\{x_i,x_j\}$, $1\leq i <j\leq |\mathcal{P}_n^{\text{m}}|$, such that $t_{i,j}\leq  g(x_i-x_j)$. For a set $\xi$ of marked points as above, we denote by $G(\xi)=(V(\xi),E(\xi))$ the soft random geometric graph constructed from $\xi$ as we just described. This construction will allow us to apply a general result on Poisson approximation from \cite{Penrose3} that we recall in the next subsection.

\subsubsection{A Poisson approximation theorem via coupling}

Let us now state the above mentioned result from \cite{Penrose3} in the particular case where the underlying Poisson point process is of intensity 1 in a bounded measurable subset $B$ of $\mathbb{R}$. Although, for our application, the mark space will be taken as above, it is interesting to note that it holds for an arbitrary marked space $\left(\mathbb{M},\mathcal{M},\mathbf{m}\right)$ endowed with a diffusive probability measure $\mathbf{m}$. We denote by $\mathbf{S}$ the set of all locally finite subsets of $\mathbb{R}\times\mathbb{M}$ and by $\mathbf{S}_k$ the set of subsets of $\mathbb{R}\times\mathbb{M}$ of cardinality $k$, $k\in\mathbb{N}$.

\begin{theorem}[{\cite[Theorem 3.1.]{Penrose3}}] \label{Th:Penrose18}

Let $k\in\mathbb{N}$, $f\colon \mathbf{S}_k\times \mathbf{S} \longrightarrow \{0,1\}$ a measurable function and for $\xi\in \mathbf{S}$ set
\begin{equation} \label{coupf}
	F(\xi):=\sum_{\psi\in\mathbf{S}_k:\psi\subset\xi}f(\psi,\xi\setminus \psi).
\end{equation}

Let $B$ be a bounded measurable subset of $\mathbb{R}$ and let $\mathcal{P}^{\text{m}}_B$ be a marked Poisson point process with intensity $\operatorname{Leb}(\cdot)\times \mathbf{m}$ in $B\times \mathbb{M}$ and set $W:=F(\mathcal{P}^{\text{m}}_B)$ and $\beta:=\mathbb{E}[W]$. For $x_1,\dots,x_k\in B$, set $$p(x_1,\dots,x_k):=\mathbb{E}\left[f\left(\overset{k}{\underset{i=1}{\bigcup}}\{(x_i,\tau_i)\},\mathcal{P}^{\text{m}}_B\right)\right],$$
where the $\tau_i$ are independent random elements of $\mathbb{M}$ with common distribution $\mathbf{m}$. 

Suppose that for almost every $\mathbf{x}=(x_1,\dots,x_k)\in B^k$ with $p(x_1,\dots,x_k)>0$ we can find coupled random variables $U_\mathbf{x}$ and $V_\mathbf{x}$ such that:
\begin{enumerate}
\item $W\overset{d}{=}U_\mathbf{x}$,
\item $F\left(\mathcal{P}^{\text{m}}_B\cup\overset{k}{\underset{i=1}{\bigcup}}\{(x_i,\tau_i)\}\right)$ conditional on $f\left(\overset{k}{\underset{i=1}{\bigcup}}\{(x_i,\tau_i)\},\mathcal{P}^{\text{m}}_B\right)=1$ has the same distribution as $1+V_\mathbf{x}$,
\item $\mathbb{E}\left[\vert U_\mathbf{x}-V_\mathbf{x}\vert\right]\leq w(\mathbf{x})$, where $w\colon B^k \to [0, \infty)$ is a measurable function.
\end{enumerate}

Let $P(\beta)$ be a mean $\beta$ Poisson random variable. Then

\begin{equation}d_{\text{TV}}\left(W,P(\beta)\right)\leq\frac{\min(1,\beta^{-1})}{k!}\int_{B^k}w(\mathbf{x})p(\mathbf{x})d \mathbf{x},\label{th:Pdtv}\end{equation}
where $d_{TV}$ denotes the total variation distance between two discrete random variables.
\end{theorem}  

\begin{remark}
Actually, \cite[Theorem 3.1]{Penrose3} also provides a bound for the Wasserstein distance between $W$ and $P(\beta)$.
\end{remark}

The quantities involved in Theorem \ref{Th:Penrose18} at first glance might seem quite abstract, but they are in close relation with the formal construction of the soft random geometric graph given in Subsection \ref{Sec:ConsSoftRGG}. For example important quantities in the graph such as the number of vertices with given degree or the number of components of given size can be studied by considering special choices of the function $F$ in \eqref{coupf}. A key observation here is that the number of edges longer than certain values can also be formulated by means of a function of the type \eqref{coupf}, which we discuss below.

\subsection{The sum of exceedances}

Theorem \ref{main} is achieved by proving the convergence of the corresponding distribution functions, which we obtain by using a Poisson approximation. To this end, following a well-known method in extreme value theory, let us introduce the number of exceedances.

Let $(r_n)_{n\in\N}$ be a sequence such that $0<r_n<2n$ for every $n\in \N$. The number (or sum) of exceedances is formally defined as the random variable $W(n,r_n)$, $n\in \N$, given by
$$  W(n,r_n)=\frac12 \sum_{x\in \mathcal{P}_n}  \sum_{y\in \mathcal{P}_n} \mathbf{1}_{\{\vert x-y \vert  \mathbf{1}_{\{\{x,y\} \in E_n\}}>r_n\}} .$$
In words, $W(n,r_n)$ is the number of edges which are longer than $r_n$. Note that the factor $\frac12$ in front of the sum is due to the fact that the number of vertices which are part of an edge of length above $r_n$ counts every exceedance twice. We show that $W(n,r_n)$, $n\in \N$, approximates a Poisson random variable by writing it as a sum of suitably defined indicator functions (without a factor $\frac12$ in front). The following are auxiliary results in proving that $W(n,r_n)$, $n\in \N$, approximates a Poisson random variable. We note that Lemma \ref{a1} and Lemma \ref{pr:continuousCase} hold for arbitrary connection functions not necessarily satisfying \eqref{cf}.
\begin{lemma}\label{a1}
	It holds
	$$ \E\left[W(n,r_n)\right] =\frac12 \int_{B_n} \int_{B_n\setminus B_{r_n}(x)} g(y-x)dydx . $$
\end{lemma}

\begin{proof}
	From the Mecke formula we recall that
	\begin{equation} \label{lp1}
		\E\left[W(n,r_n)\right] =\frac12 \int_{B_n} \int_{B_n}   \P_{x,y} \left( \{x,y\}\in E^{x,y}_n, |x-y| >r_n \right) dx.
	\end{equation}
	Here $\P_{x,y}$ denotes the distribution of the Poisson point process with two points added at $x,y$ and $E^{x,y}_n$ the edge-set of the soft random geometric graph constructed with two points added at $x,y$. Since $x$ and $y$ are connected with probability $g(x-y)$, the result follows.
\end{proof}

\begin{lemma}\label{pr:continuousCase} Let $P(n)$ be a Poisson random variable with mean $\E\left[W(n,r_n)\right].$ It holds that
	\begin{equation*}d_{\text{TV}}\left(W(n,r_n),P(n)\right)\leq \min(1,\E\left[W(n,r_n)\right]^{-1})\mathfrak{I}(n),\end{equation*}
	where
	\begin{align*}
	     \mathfrak{I}(n)&=\int_{B_n}\int_{B_n\setminus B_{r_n}(x_1)}\Big[ \int_{B_n\setminus B_{r_n}(x_1)} g(x_1-z) dz + \int_{B_n\setminus B_{r_n}(x_2)} g(x_2-z) dz\Big]  \\
      & \quad \times g(x_1 -x_2) dx_1 dx_2.
	\end{align*}

\end{lemma}

\begin{proof}
By using the construction of the soft random geometric graph in terms of a marked Poisson point process described in Subsection \ref{Sec:ConsSoftRGG}, we can apply Theorem \ref{Th:Penrose18} with $k=2$ and $f \colon \mathbf{S}_2\times \mathbf{S}$, where $\mathbf{S}_2$ is the set of all subsets of $ B_n\times \mathbb{M}$ with size 2, defined by
\begin{align*}
    &f\left( \left\{\left( x,(u_i)_{i\geq 0}\right), ( y,(v_i)_{i\geq 0})\right\},\xi \right) =\mathbf{1}_{ \big\{ \{x,y\} \in E(\xi\cup\{(x,(u_i)_{i\geq 0}), ( y,(v_i)_{i\geq 0}) \}), \vert x-y \vert >r_n, y \geq x \big\}} 
\end{align*}
for $\{( x,(u_i)_{i\geq 0}), ( y,(v_i)_{i\geq 0}) \} \in \mathbf{S}_2$, so that $ W=F(\mathcal{P}_n^{\text{m}})=W(n,r_n)$. One can formally check that $f$ is the indicator function of a measurable set, see \cite[Section 3]{rs} for details.

In order to apply Theorem \ref{Th:Penrose18}, we now need to define the coupled random variables $U_{x_1,x_2}$ and $V_{x_1,x_2}$. To this end, let us enlarge $\mathcal{P}_n^{\text{m}}$ by adding two marked points $x_1^{\text{m}}=(x_1,(u_i)_{i\geq 0})$ and $x_2^{\text{m}}=(x_2,(v_i)_{i\geq 0})$. More precisely, we add two points $x_1$ and $x_2$ in the ground process $\mathcal{P}_n$ of $\mathcal{P}_n^{\text{m}}$ and we attach to them random marks $(u_i)_{i\geq 0}$ and $(v_i)_{i\geq 0}$ distributed according to $\mathbf{m}$ and independent of $\mathcal{P}_n^{\text{m}}$. Then  we define $V_{x_1,x_2}$ as the number of exceedances in the enlarged graph $G(\mathcal{P}^{\text{m}}_n \cup \{x_1^{\text{m}}, x_2^{\text{m}}\} )$ other than the one at $x_1$ and $x_2$ (if there is one), more precisely
\begin{align*}
 V_{x_1,x_2}&= \sum_{y,z \in \mathcal{P}_n \cup \{x_1, x_2\}}  
\mathbf{1}_{ \big\{ \{y,z\} \in E(\mathcal{P}^{\text{m}}_n\cup\{(x_1,(u_i)_{i\geq 0}), ( x_2,(v_i)_{i\geq 0}) \}), \vert z-y \vert >r_n, z \geq y \big\}} \\
& \quad - \mathbf{1}_{ \big\{ \{x_1,x_2\} \in E(\mathcal{P}^{\text{m}}_n\cup\{(x_1,(u_i)_{i\geq 0}), ( x_2,(v_i)_{i\geq 0}) \}), \vert x_1-x_2 \vert >r_n, x_2 \geq x_1 \big\}}
\end{align*}
 .

Since the point process is simple, it follows directly from this definition and the ones of $f$ and $F$ that almost surely
\[F\left(\mathcal{P}^{\text{m}}_n \cup \{x_1^{\text{m}}, x_2^{\text{m}}\}\right)= V_{x_1,x_2}+f\left(\left\{\left(x_1,(u_i)_{i\geq 0}\right), (x_2,(v_i)_{i\geq 0}\right)\right\}, \mathcal{P}_n^{\text{m}}).\]
In particular, $V_{x_1,x_2}$ satisfies Assumption {\it (2)} of Theorem \ref{Th:Penrose18}.

Now  let $H_n=(\mathcal{P}_n,E_{H_n})$ be the subgraph of $G(\mathcal{P}^{\text{m}}_n \cup \{x_1^{\text{m}},x_2^{\text{m}}\})$ obtained from it by deleting the vertices $x_1$, $x_2$ and all the edges having $x_1$ or $x_2$ as an endpoint. The random variable $U_{x_1,x_2}$ is then defined as the number of exceedances in $H_n$, namely:
$$ U_{x_1,x_2}= \sum_{y,z \in \mathcal{P}_n} \mathbf{1}_{ \big\{ \{y,z\} \in E_{H_n}, \vert z-y \vert >r_n, z \geq y \big\}}.$$ 

Note that $H_n$ does not coincide a.s.\,with $G(\mathcal{P}^{\text{m}}_n)$. Indeed, an edge between two points $x$ and $y$ in $\mathcal{P}_n$ does not only depend on the marks at $x$ and $y$ but also on all the other marks, since they are used to determine the order in which the points are enumerated (see Subsection \ref{Sec:ConsSoftRGG}). Nevertheless, one can see that $H_n$ has the same distribution as $G(\mathcal{P}^{\text{m}}_n)$, since the marks are independent sequences of independent uniform random variables on $[0,1)$. In particular, $U_{x_1,x_2}$ has the same distribution as $W$; in other words, Assumption {\it (1)} of Theorem \ref{Th:Penrose18} is satisfied.

Let us now check Assumption {\it (3)} of Theorem \ref{Th:Penrose18}. In the following we will write $x \leftrightarrow y$ if two vertices are connected by an edge. Since $V_{x_1,x_2}\geq U_{x_1,x_2}$ by construction, it holds that
\begin{align*}
& \mathbb{E}\left[\vert U_{x_1,x_2}-V_{x_1,x_2}\vert\right]=\mathbb{E}\left[ V_{x_1,x_2}-U_{x_1,x_2}\right]\\
& \leq \mathbb{E}\left[\sum_{y \in \{x_1,x_2\} } \sum_{z \in \mathcal{P}_n} \mathbf{1}_{ \big\{ \{y,z\} \in E(\mathcal{P}^{\text{m}}_n \cup \{x_1^{\text{m}}, x_2^{\text{m}}\}), \vert z-y \vert >r_n, z \geq y \big\}} \right]\\
&\leq  \mathbb{E}\left[ \sum_{z \in \mathcal{P}_n} \mathbf{1}_{ \big\{ \{x_1,z\} \in E(\mathcal{P}^{\text{m}}_n \cup \{x_1^{\text{m}}, x_2^{\text{m}}\}), \vert z-x_1 \vert >r_n \big\}} \right] + \mathbb{E}\left[ \sum_{z \in \mathcal{P}_n} \mathbf{1}_{ \big\{ \{x_2,z\} \in E(\mathcal{P}^{\text{m}}_n \cup \{x_1^{\text{m}}, x_2^{\text{m}}\}), \vert z-x_2 \vert >r_n \big\}} \right] \\
&=\int_{B_n} \mathbb{P}_z\left(|x_1-z| \mathbf{1}_{\{z\leftrightarrow x_1\}} > r_n\right)dz + \int_{B_n} \mathbb{P}_z\left(|x_2-z| \mathbf{1}_{\{z\leftrightarrow x_2\}} > r_n\right)dz\nonumber\\
&=\int_{B_n\setminus B_{r_n}(x_1)} g(x_1-z) dz + \int_{B_n\setminus B_{r_n}(x_2)} g(x_2-z) dz=:w(x_1, x_2) ,
\end{align*}
where we used the Mecke formula. So Assumption {\it (3)} of Theorem \ref{Th:Penrose18} is also satisfied. 

Finally, we rewrite the integral in the r.h.s. of \eqref{th:Pdtv} as follows:
\begin{align*}
&\int_{B_n} \int_{B_n}w(x_1,x_2)p(x_1,x_2)dx_1dx_2
\\&=\int_{B_n}\int_{B_n}\Big[ \int_{B_n\setminus B_{r_n}(x_1)} g(x_1-z) dz + \int_{B_n\setminus B_{r_n}(x_2)} g(x_2-z) dz\Big] \\
& \quad \times \mathbb{E}\left[f\left(\{x_1^{\text{m}}, x_2^{\text{m}}\},\mathcal{P}^{\text{m}}_B\right)\right]dx_1 dx_2\\
&=\int_{B_n}\int_{B_n}\Big[ \int_{B_n\setminus B_{r_n}(x_1)} g(x_1-z) dz + \int_{B_n\setminus B_{r_n}(x_2)} g(x_2-z) dz\Big] \\
& \quad \times \mathbb{P}_{x_1,x_2}\left( \max_{y\in\mathcal{P}_n} \vert x_1-x_2\vert\mathbf{1}_{\{x_1\leftrightarrow x_2\}} >r_n\right) dx_1 dx_2\\
&=\int_{B_n}\int_{B_n\setminus B_{r_n}(x_1)}\Big[ \int_{B_n\setminus B_{r_n}(x_1)} g(x_1-z) dz + \int_{B_n\setminus B_{r_n}(x_2)} g(x_2-z) dz\Big]  g(x_1 -x_2) dx_1 dx_2\\
&=\mathfrak{I}(n).
\end{align*}
This concludes the proof.
\end{proof}

We now state a result derived from Lemma \ref{pr:continuousCase}. Since it follows from Lemma~\ref{pr:continuousCase} and rather standard computations, we postpone the details of its proof to the Appendix. On the one hand this result yields the Poisson approximations needed to complete the proof of Theorem \ref{main}. On the other hand it additionally illustrates some subtle differences regarding the upper bounds one obtains for the convergence speed. 


\begin{prop}\label{prop:CVP}

	Let $r\in (0,1)$. Define $g_r = \exp \left( - \sqrt{-8 \ln r} \right)$, $c_r=h^{-1}\left(r^{-\frac{1}{2}}\right)$, where $h$ is as in Theorem \ref{main} {\it (II)}, and for $n\in\mathbb{N}$ 
 \begin{equation}\label{eq:defrn}
r_n:=\left\{\begin{array}{ll}
(2n)^{\frac{1}{\alpha-1}} \left( \frac{1-\alpha}{2} \ln r \right)^{\frac{1}{1-\alpha}},&\mbox{if }\alpha \in (2,\infty),\\
c_rn,&\mbox{if }\alpha =2,\\
\left( (2n)^{2-\alpha} - (2-\alpha) 2^{1 - \frac{\alpha}{2}} \sqrt{-\ln r} n^{1-\frac{\alpha}{2}} \right)^{\frac{1}{2-\alpha}},&\mbox{if } \alpha \in (1,2),\\
2n g_r^{\frac{1}{4\sqrt{n}}},&\mbox{if }\alpha =1,\\
\left( (2n)^{1-\alpha} - (1-\alpha) 2^{ - \frac{\alpha}{2}} \sqrt{-\ln r} n^{-\frac{\alpha}{2}} \right)^{\frac{1}{1-\alpha}},&\mbox{if } \alpha \in (0,1).
\end{array}\right.
\end{equation}
   
   Then, for each $\alpha \in (0, \infty)$, there exists a constant $C \in (0,\infty)$ independent of $n$ such that
	$$ d_{\text{TV}}\left(W(n,r_n),P(n)\right)\leq \left\{ \begin{array}{ll}
	    C n^{1-\alpha}, &\mbox{if }\alpha\in (1,\infty),  \\
	     C n^{-\frac{\alpha}{2}},& \mbox{if }\alpha\in (0,1],
	\end{array}\right.$$
 where $P(n)$ is a Poisson random variable with mean $\mathbb{E}[W(n,r_n)]$. Moreover, in all the cases, it holds that 
	$$ \lim_{n \to \infty} \E\left[W(n,r_n)\right] = -\frac12 \ln r.$$
\end{prop}
\begin{remark}
The following is a purely mathematical observation, which seems interesting for us. For fixed $r$ and $n$, the thresholds $r_n$ in \eqref{eq:defrn}, seen as a function of {\color{black}$\alpha\in (0,\infty)$, are continuous except at 1 and 2, precisely where the phase transitions occur}. Though the random graphs are more irregular for $\alpha \in (0,1]$, as seen in our results and as they approximate locally infinite graphs, it appears that they exhibit mathematically interesting structures. 
\end{remark}
{\color{black}\subsection{Conclusion} 
For $r\in (0,1)$ and $\alpha \in (0, \infty)$ fixed, it follows from Proposition~\ref{prop:CVP} that $W(n,r_n)$ converges in distribution to a Poisson random variable with mean $-\ln r$, as $n\rightarrow \infty$. In particular, since $\{e^*_n\leq r_n\}=\{W(n)=0\}$, we obtain the following corollary. 

\begin{cor}
Let $\alpha \in (0, \infty)$ and $r\in (0,1)$. It holds that
\[\lim_{n\to \infty}\P \left(e^*_n\leq r_n\right)=r^{\frac12}.\]
\end{cor}
Since the event $\{e^*_n\leq r_n \}$ can be rewritten as the event $\{f_n(e^*_n)\leq r^{\frac12} \}$, where
\[ f_n(e^*_n) = \left \{\begin{array}{ll}
      \exp \left( \frac{2n}{1-\alpha}    (e_n^*)^{1-\alpha} \right), &\mbox{if }\alpha \in (2,\infty),\\
      h \left( \frac{e_n^*}{n}\right)^{-1} , &\mbox{if }\alpha=2,\\
      \exp \left( -(2n)^{\alpha-2} 2^{-1}(2-\alpha)^{-2}  \left( (e_n^*)^{2-\alpha} - (2n)^{2-\alpha} \right)^2 \right) , &\mbox{if }\alpha\in (1,2) ,\\
      \exp\left( -n \left( \ln\frac{e_n^*}{2n} \right) ^{2} \right), &\mbox{if }\alpha=1,\\
      \exp \left( -(2n)^\alpha  2^{-1} (1-\alpha)^{-2}  \left( (e_n^*)^{1-\alpha} - (2n)^{1-\alpha} \right)^2 \right) , &\mbox{if }\alpha\in (0,1),
\end{array}\right.\]
Theorem \ref{main} readily follows.
}



\appendix

\section{Proof of Proposition \ref{prop:CVP}}

The following fact will be useful in applying Lemma \ref{pr:continuousCase}. Since $g(x){\sim}x^{-\alpha}$, $x\to \infty$, it is a consequence of Lemma \ref{a1}. If $r_n$ is so that $r_n\to \infty$, as $n \to \infty$, and
\begin{equation}\label{eq:NC3}
\max_{x\in B_n}\int_{B_n\setminus B_{r_n}(x)} g(y-x) dy \underset{n\to \infty}{\longrightarrow} 0
\end{equation} 
then 
$$\E[ W(n,r_n) ]\underset{n\to \infty}{\sim} \frac12 I(n),$$
with
$$I(n):=\int_{B_n}\int_{B_n\setminus B_{r_n}(x)} |y-x|^{-\alpha}dy dx$$
and 
\begin{align*}
\mathfrak{I}(n)& \leq 2  \left( \max_{x\in B_n}\int_{B_n\setminus B_{r_n}(x)} g(y-x) dy \right) \int_{B_n} \int_{B_n\setminus B_{r_n}(x_1)} g(x_1-x_2) dx_2 dx_1\\
&\leq C \E\left[W(n,r_n)\right]\max_{x\in B_n}\int_{B_n\setminus B_{r_n}(x)} g(y-x) dy.
\end{align*}
\subsection{Case $\alpha>2$}
	We first check that
	$$ \lim_{n\to\infty} I(n) =  -\ln r.$$
	We write
	\begin{align*}
	I(n)=2\int_{0}^n \int_{B_n \setminus B_{r_n}(x)} |y-x|^{-\alpha} dy dx = 2I_1 (n) + 2I_2 (n)
	\end{align*}
	with
	\begin{align*}
	I_1(n) & = \int_{n-r_n}^n \int^{x-r_n}_{-n} |y-x|^{-\alpha} dy dx=\int_{n-r_n}^n \int^{n+x}_{r_n} s^{-\alpha} ds dx\\
	&=\frac{1}{1-\alpha}\int_{n-r_n}^n (n+x)^{1-\alpha}-{r_n}^{1-\alpha} dx\\
	& = \frac{1}{1-\alpha}  \frac{1}{2-\alpha} \left( (2n)^{2-\alpha} - \left(2n-r_n\right)^{2-\alpha} \right) - \frac{1}{1-\alpha} r_n^{2-\alpha}
	\end{align*}
	and
	\begin{align*}
	I_2(n) & = \int_{0}^{n-r_n} \left(\int^{x-r_n}_{-n} |y-x|^{-\alpha} dy  +  \int_{x+r_n}^{n} |y-x|^{-\alpha} dy\right) dx\\
	& =\int_{0}^{n-r_n} \left(\int^{n+x}_{r_n} s^{-\alpha} ds  +  \int_{r_n}^{n-x} s^{-\alpha} ds\right) dx\\
	& =\frac{1}{1-\alpha}\int_{0}^{n-r_n} \left(\left(n+x\right)^{1-\alpha}+\left(n-x\right)^{1-\alpha}-2r_n^{1-\alpha}\right) dx\\
	& =\frac{1}{1-\alpha}\frac{1}{2-\alpha}\left(\left(2n-r_n\right)^{2-\alpha}-r_n^{2-\alpha}\right)-\frac{2}{1-\alpha}(n-r_n)r_n^{1-\alpha}.
	\end{align*}
	It follows that
	\begin{align*}
	I(n)&=\frac{2}{1-\alpha}\left(\frac{1}{2-\alpha}\left(\left(2n\right)^{2-\alpha}-r_n^{2-\alpha}\right)-(2n-r_n)r_n^{1-\alpha}\right)\\
	&=\frac{2}{1-\alpha}\Bigg[\frac{1}{2-\alpha}\left(\left(2n\right)^{2-\alpha}-(2n)^{\frac{2-\alpha}{\alpha-1}} \left( \frac{1-\alpha}{2} \ln r \right)^{\frac{2-\alpha}{1-\alpha}}\right)\\
	&\qquad\qquad\qquad -\left(1-(2n)^{-\frac{\alpha}{\alpha-1}} \left( \frac{1-\alpha}{2} \ln r \right)^{\frac{1}{1-\alpha}}\right)\left( \frac{1-\alpha}{2} \ln r \right)\Bigg]\\
	&\underset{n\to \infty}{\longrightarrow} -\ln r,
	\end{align*}
	since $\alpha>2$. Thus, the claim follows from Lemma \ref{a1} and \eqref{eq:NC3}. The assertion \eqref{eq:NC3} is easily verified. Then from Lemma \ref{pr:continuousCase} and straightforward calculations we get
	\begin{align*}
		& d_{\text{TV}}\left(W(n,r_n),P(n)\right) \\
		& \leq 2\left( \max_{x\in B_n} \int_{B_n\setminus B_{r_n}(x)}g(y-x)d y\right)  \min(1,\E\left[P(n)\right]^{-1}) \int_{B_n} \int_{B_n\setminus B_{r_n}(x)}g(y-x)d yd x \\
		& \leq Cn^{1-\alpha}.
	\end{align*}
	
\subsection{Case $\alpha=2$}
	Again \eqref{eq:NC3} is easily verified and the upper bound on the total variation distance follows as in the previous regime, once we check the second assertion of this Theorem. Note that the function
	$h\colon (0,2] \to [1,\infty)$ is bijective. For notational simplicity define $c_r=h^{-1}\left(r^{-\frac{1}{2}}\right)\in (0,2]$. First, assume that $r$ is such that $c_r\in (0,1]$. It follows that
	\begin{align*}
	I(n)&=2\int_{0}^n \int_{B_n \setminus B_{r_n}(x)} |y-x|^{-2} dy dx\\
	& = 2\int_{0}^n \int_{c_rn}^{n+x} s^{-2} ds dx+ 2\int_{0}^{(1-c_r)n} \int_{c_rn}^{n-x} s^{-2} ds dx\\
	& = 2\int_{0}^n ((c_rn)^{-1}-(n+x)^{-1}) dx+ 2\int_{0}^{(1-c_r)n} ((c_rn)^{-1}-(n-x)^{-1}) dx\\
	& = 2 \frac{2-c_r}{c_r}- 2 \int_{0}^n (n+x)^{-1} dx- 2\int_{0}^{(1-c_r)n} (n-x)^{-1} dx\\
	& = 2 \frac{2-c_r}{c_r}- 2 \ln(2n)+2\ln(n)+ 2\ln (c_rn)-2\ln (n)\\	
	&=2\left(\ln\left(\exp\left(\frac{2-c_r}{c_r}\right)\right)- \ln(2n)+\ln (c_rn)\right)\\
	&=2\ln\left(\frac{c_r}{2}\exp\left(\frac{2-c_r}{c_r}\right)\right)=-\ln\left(\left(\frac{c_r}{2}\exp\left(\frac{2-c_r}{c_r}\right)\right)^{-2}\right)\\	
	&=-\ln r
	\end{align*}
	by definition of $c_r$.
	
	Assume now that $r$ is such that $c_r\in (1,2]$. Then it follows that
	\begin{align*}
	I(n)&=\int_{-n}^n \int_{B_n \setminus B_{r_n}(x)} |y-x|^{-2} dy dx = 2\int_{0}^n \int_{B_n \setminus B_{r_n}(x)} |y-x|^{-2} dy dx\\ 
	& = 2\int_{(c_r-1)n}^n \int_{c_rn}^{n+x} s^{-2} ds dx = 2\int_{(c_r-1)n}^n ((c_rn)^{-1}-(n+x)^{-1}) dx\\
	& = 2 \left(\frac{2-c_r}{c_r}-\ln (2n)+\ln(c_rn)\right)=-\ln r
	\end{align*}
	by definition of $c_r$ as before.
	
\subsection{Case $1<\alpha<2$}
	Since $\alpha>1$, \eqref{eq:NC3} easily follows. Note that, for $n$ large, $n<r_n<2n$. We have
	
	\begin{align}
	I(n)&=2\int_{0}^n \int_{B_n \setminus B_{r_n}(x)} |y-x|^{-\alpha} dy dx = 2\int_{-n+r_n}^n \int_{B_n \setminus B_{r_n}(x)} |y-x|^{-\alpha} dy dx\nonumber\\
	&= 2\int_{-n+r_n}^n \int_{r_n}^{n+x}s^{-\alpha} ds dx=\frac{2}{1-\alpha}\int_{-n+r_n}^n \left(n+x\right)^{1-\alpha}-r_n^{1-\alpha}dx\nonumber\\
	&=\frac{2}{1-\alpha}\left(\frac{1}{2-\alpha}\left((2n)^{2-\alpha}-r_n^{2-\alpha}\right)-(2n-r_n)r_n^{1-\alpha}\right)\nonumber\\
	&=\frac{2}{1-\alpha}\left((2n)^{1-\alpha}-r_n^{1-\alpha}\right)2n-\frac{2}{2-\alpha}\left((2n)^{2-\alpha}-r_n^{2-\alpha}\right)\nonumber\\	
	&=:\frac{2}{1-\alpha} \widetilde{I}(n)\label{eq:tile{I}}
	\end{align}
	with 
	\begin{align*} 
	\begin{split}
	\widetilde{I}(n) & = \left( (2n)^{1-\alpha} - \left( (2n)^{2-\alpha} - (2-\alpha) 2^{1 - \frac{\alpha}{2}} \sqrt{-\ln r} n^{1-\frac{\alpha}{2}} \right)^{\frac{1-\alpha}{2-\alpha}} \right) 2n\\&
	\qquad\qquad - (1-\alpha) 2^{1 - \frac{\alpha}{2}} \sqrt{-\ln r} n^{1-\frac{\alpha}{2}}  \\
	& = n^{1-\frac{\alpha}{2}} \Bigg[ -(1-\alpha) 2^{1-\frac{\alpha}{2}} \sqrt{-\ln r} \\
	& \qquad\qquad+ 2^{2-\alpha} n^{1-\frac{\alpha}{2}}\left( 1- \left( 1 - (2-\alpha) 2^{ 1- \frac{\alpha}{2}} \sqrt{-\ln r} 2^{-2+\alpha}n^{-1+\frac{\alpha}{2}} \right)^{\frac{1-\alpha}{2-\alpha}}  \right) \Bigg].
	\end{split}
	\end{align*}

	By using l'H\^{o}pital's rule, one can see that
	\begin{align*}
	&\lim_{n \to \infty} 2^{2-\alpha}  n^{1-\frac{\alpha}{2}} \left(1- \left( 1 - (2-\alpha) 2^{1 - \frac{\alpha}{2}} \sqrt{-\ln r}2^{-2+\alpha}n^{-1+\frac{\alpha}{2}} \right)^{\frac{1-\alpha}{2-\alpha}}  \right) \\
	& = (1-\alpha) 2^{ 1- \frac{\alpha}{2}} \sqrt{-\ln r}
	\end{align*}
	and then conclude by using l'H\^{o}pital's rule again that
	\begin{align*}
	\lim_{n \to \infty} I(n) = \frac{2}{1-\alpha} \lim_{n \to \infty}\widetilde{I}(n) = -\ln r.
	\end{align*}
	The conclusion of this Theorem then follows the same arguments as before.
\subsection{Case $\alpha=1$}
 First, we verify \eqref{eq:NC3}. Observe that, for $x\in[0,r_n-n)$, $B_n\setminus B_n(x)=\emptyset$. Besides, for $x \in [r_n-n,n]$, we have
	\begin{align*}
	\int_{B_n \setminus B_{r_n}(x)} |y-x|^{-\alpha} dy & \leq r_n^{-\alpha}(x-r_n+n) \leq r_n^{-\alpha}(2n-r_n). 
	\end{align*}
	Hence,
	\begin{align*}
	\max_{x\in [0,n]}\int_{B_n \setminus B_{r_n}(x)} |y-x|^{-\alpha} dy & \leq 2r_n^{-\alpha}(2n-r_n)\underset{n\to\infty}{\longrightarrow} 0,
	\end{align*}
	which proves \eqref{eq:NC3}.
	
	For $n$ large, it holds that
	\begin{align*}
	I(n)&=2\int_{0}^n \int_{B_n \setminus B_{r_n}(x)} |y-x|^{-1} dydx = 2\int_{-n+r_n}^{n} \int_{x-r_n}^{-n} |y-x|^{-1} dy dx \\
	& = 2\int_{-n+r_n}^{n} \int_{r_n}^{n+x} s^{-1} ds dx = 2\int_{-n+r_n }^{n} \ln (n+x) - \ln r_n dx\\
	&=2\left( 2n\ln(2n)-n-r_n\ln(r_n)-n+r_n-(2n-r_n)\ln r_n\right)\\
	&= -4n \ln  \left( \frac{r_n}{2n}\right)  +2r_n -4n= - \sqrt{n} \ln g_r + 4n \left( g_r^{\frac{1}{4\sqrt{n}}} -1 \right)\\ 
	&= \frac{ 4\sqrt{n} \left( g_r^{\frac{1}{4\sqrt{n}}} -1 \right) - \ln g_r}{n^{-\frac{1}{2}}}.
	\end{align*}
	Since
	$$4\sqrt{n} \left( g_r^{\frac{1}{4\sqrt{n}}} -1 \right) \underset{n\to\infty}{\longrightarrow} \ln g_r,$$
	we use once again l'H\^{o}pital's rule to check that
	\begin{align*}
	\lim_{n\to\infty} I(n)	&= \lim_{n \to \infty} \frac{ 4\sqrt{n} \left( g_r^{\frac{1}{4\sqrt{n}}} -1 \right) - \ln g_r}{n^{-\frac{1}{2}}}
	= \frac{1}{8}  (\ln g_r)^2= -\ln r.
	\end{align*}
	Thus, from \eqref{eq:NC3} and Lemma \ref{a1} we get
	$$ \lim_{n \to \infty} \E\left[W(n,r_n)\right] = -\frac12 \ln r.$$
	Finally, Lemma \ref{pr:continuousCase} gives us
		\begin{align*}
	& d_{\text{TV}}\left(W(n,r_n),P(n)\right) \\
		& \leq 2\left( \max_{x\in B_n} \int_{B_n\setminus B_{r_n}(x)}g(y-x)d y\right)  \min(1,\E\left[P(n)\right]^{-1}) \int_{B_n} \int_{B_n\setminus B_{r_n}(x)}g(y-x)d yd x \\
	& \leq C' \left(\frac{2n}{r_n}-1\right) = C' \left( \exp\left( \sqrt{-\frac{1}{2n} \ln r} \right) -1\right) \leq C n^{-\frac12}.
	\end{align*}
\subsection{Case $0<\alpha<1$}

	We begin by verifying \eqref{pr:continuousCase}. Indeed
	\begin{align*}
	\max_{x\in[0,n]} \int_{B_n \setminus B_{r_n}(x)} |y-x|^{-\alpha} dy  &\leq 2\max_{x \in [r_n-n,n]} \int_{r_n \wedge (n+x)}^{n+x} s^{-\alpha} ds \\&\leq \max_{x \in [r_n-n,n]} \frac{2}{1-\alpha} \left( (n+x)^{1-\alpha} - r_n^{1-\alpha} \right) \\
	& = \frac{2}{1-\alpha} \left( (2n)^{1-\alpha} - r_n^{1-\alpha} \right) = 2^{ 1- \frac{\alpha}{2}} \sqrt{-\ln r} n^{-\frac{\alpha}{2}}\\ 
	& \underset{n\to\infty}{\longrightarrow} 0.
	\end{align*}
	Now we prove that
	$$ \lim_{n \to \infty} \E\left[W(n,r_n)\right] = -\frac12 \ln r.$$
	The calculations are very close to the ones in the case $1<\alpha<2$. Equation \eqref{eq:tile{I}} is satisfied here with
	\begin{align*} \label{two}
	\begin{split}
	\widetilde{I}(n) & = n^{1-\frac{\alpha}{2}} (2-\alpha) 2^{1-\frac{\alpha}{2}} \sqrt{-\ln r} \\
	&\qquad + \left( (2n)^{1-\alpha} - (1-\alpha) 2^{ - \frac{\alpha}{2}} \sqrt{-\ln r} n^{-\frac{\alpha}{2}} \right)^{\frac{2-\alpha}{1-\alpha}} -(2n)^{2-\alpha} \\
	& = n^{1-\frac{\alpha}{2}} \Bigg[ (2-\alpha) 2^{1-\frac{\alpha}{2}} \sqrt{-\ln r} \\
	& \qquad\qquad+ 2^{2-\alpha} n^{1-\frac{\alpha}{2}}\left( \left( 1 - (1-\alpha) 2^{ 1- \frac{\alpha}{2}} \sqrt{-\ln r} 2^{-2+\alpha}n^{-1+\frac{\alpha}{2}} \right)^{\frac{2-\alpha}{1-\alpha}} -1 \right) \Bigg].
	\end{split}
	\end{align*}
	From here one can conclude by using l'H\^opital's rule twice again. Finally, the first assertion follows from similar calculations as in the case $\alpha=1$. This concludes the proof.

\subsection*{Acknowledgement} The IMB receives support from the EIPHI Graduate
School (contract ANR-17-EURE-0002).
\bibliographystyle{plain}
\bibliography{lit}

\end{document}